\documentclass{amsart}
\usepackage[utf8]{inputenc}
\usepackage{amsfonts}
\usepackage{amsmath}
\usepackage{amsthm}
\usepackage{amssymb}
\usepackage{latexsym}
\usepackage{enumerate}
\usepackage{centernot}

\newtheorem{theorem}{Theorem}
\newtheorem{lemma}[theorem]{Lemma}
\newtheorem{proposition}[theorem]{Proposition}
\newtheorem{corollary}[theorem]{Corollary}

\newtheorem{remark}[theorem]{Remark}

\newtheorem{question}[theorem]{Question}
\newtheorem{example}[theorem]{Example}
\newtheorem{definition}[theorem]{Definition}
\newtheorem*{claim}{Claim}
\newtheorem*{nono-theorem}{Theorem}

\newcommand{\F}{\mathcal{F}}
\newcommand{\G}{\mathcal{G}}
\newcommand{\GM}{\mathcal{G}^{MAX}}
\newcommand{\Sh}{\mathcal{S}}
\newcommand{\FM}{\mathcal{F}^{MAX}}

\newcommand{\N}{\mathbb{N}}
\newcommand{\supp}{\mathrm{supp\, }}
\DeclareMathOperator{\fin}{[\omega]^{<\omega}}
\DeclareMathOperator{\rk}{rk}

\begin{document}
\title{Banach-Stone-like results for combinatorial Banach spaces}

\author{C. Brech}
\address{Departamento de Matem\'atica, Instituto de Matem\'atica e Estat\'\i stica, Universidade de S\~ao Paulo, Rua do Mat\~ao, 1010 - CEP 05508-090 - S\~ao Paulo - SP - Brazil.}
\email{brech@ime.usp.br}

\author{C. Pi\~{n}a}
\address{Departamento de Matem\'aticas, Facultad de Ciencias, Universidad de los Andes, C.P. 5101,  Mérida- Venezuela.\newline
\indent Escuela de Matem\'aticas, Universidad Industrial de Santander, C.P. 680001, Bucaramanga - Colombia.}
\email{cpinarangel@gmail.com}

\thanks{The first author was supported by CNPq grant (308183/2018-5) and by FAPESP grant (2016/25574-8). The second author was supported by the postdoctoral program of Vicerrector\'{i}a de Investigaci\'{o}n y Extensi\'{o}n de la UIS}

\begin{abstract}
We show that under a certain topological assumption on two compact hereditary families $\F$ and $\G$ on some infinite cardinal $\kappa$, the corresponding combinatorial spaces $X_\F$ and $X_\G$ are isometric if and only if there is a permutation of $\kappa$ inducing a homeomorphism between $\F$ and $\G$. We also prove that two different regular families $\F$ and $\G$ on $\omega$ cannot be permuted one to the other. Both these results strengthen the main result of \cite{BrechFerencziTcaciuc}.    
\end{abstract}

\subjclass{Primary 46B45; Secondary 03E75}
\keywords{Banach-Stone Theorem, combinatorial Banach spaces, compact and spreading families.}

\maketitle
\date{\today}

\baselineskip 18pt

The classical Banach-Stone theorem states that, given two compact spaces $K$ and $L$, $C(K)$ and $C(L)$ are isometric Banach spaces if and only if $K$ and $L$ are homeomorphic (see e.g. \cite[Theorem 7.8.4]{Semadeni}). Our main purpose in this paper is to prove versions of this result in the context of combinatorial Banach spaces.

Given a family $\F$ of finite subsets of an infinite cardinal $\kappa$, it can be seen as a topological subspace of $2^\kappa$. If $\F$ contains all singletons of $\kappa$ and is closed under subsets, the combinatorial Banach space $X_\F$ (also called the generalized Schreier space) is the completion of $c_{00}(\kappa)$, the vector space of finitely supported scalar sequences, with respect to the norm:
$$
\Vert x\Vert=\sup\left\{\sum_{\alpha\in s}|x(\alpha)|: s\in\F\right\}.
$$
Notice that the sequence of unit vectors $(e_\alpha)$ forms a (long) Schauder basis of $X_\F$, which is considered to be the canonical one and is $1$-unconditional. Moreover, if $\F$ is compact, the canonical basis is shrinking. 

Both spaces $X_\F$ and $C(\F)$ can be seen as Banach spaces counterparts of the family $\F$ and they share some properties: for instance, both contain homeomorphic copies of $\F$ in the dual unit sphere equiped with the weak$^*$ topology. Banach-Stone theorem guarantees that the isometric structure of $C(\F)$ carries exactly the topological structure of $\F$. In Section 2 we prove our first main result, Theorem \ref{perm}, and give in Corollary \ref{coroperm} sufficient topological assumptions on two families $\F$ and $\G$, so that $X_\F$ and $X_\G$ are isometric if and only if $\F$ and $\G$ are $\pi$-homeomorphic (that is, there is a homeomorphism between $\F$ and $\G$ which is induced by a permutation of $\kappa$, see Section 1). Hence, the isometric structure of $X_\F$ carries more than the topological structure of $\F$.

Isometries between combinatorial spaces have been recently analysed in \cite{AntunesBeanlandVietChu, BrechFerencziTcaciuc}. In \cite[Theorem 5.1]{AntunesBeanlandVietChu}, the authors prove that any isometry of the real combinatorial space of a Schreier family of finite order is determined by a change of signs of the elements of the canonical basis. This has been generalized in \cite{BrechFerencziTcaciuc} to real and complex combinatorial spaces of regular families on the set of the integers $\omega$ (see Definition \ref{defFamiliesN}). Our proof follows similar lines as theirs, but we extracted a sufficient topological condition which made it possible to enlarge the class of families for which any isometry between two combinatorial spaces is given by a signed permutation of the canonical bases: it includes now families on $\omega$ which are not spreading, as well as families on uncountable cardinals $\kappa$. Moreover, Arens and Kelley \cite{ArensKelley} gave a simple proof of the Banach-Stone theorem using the fact that isometries take extreme points to extreme points. It was improved to a noncommutative version for C$^*$-algebras by Kadison \cite{Kadison}. Our proof also uses this fact, as does that of \cite{BrechFerencziTcaciuc}.

Theorem \ref{perm} states that any isometry between $X_\F$ and $X_\G$ is induced by a permutation of $\kappa$ and a sequence of signs for a class of families. If we take $\F$ to be the family of singletons of $\omega$, we get that $X_\F = c$, the Banach space of convergent sequences of scalars, with the supremum norm. Taking $\F$ to be the family of all finite subsets of $\kappa$ makes $X_\F$ to be $\ell_1(\kappa)$. Hence, similarly to what is mentioned in \cite{BrechFerencziTcaciuc}, our results can be compared to the fact that isometries of the spaces $c_0$ or $\ell_p$, $1 \leq p < \infty$, $p\neq 2$, are all signed permutations of the canonical unit basis (see e.g. 
\cite[Theorem 2.f.14]{LindenstraussTzafriri}).

Our second main result, Theorem \ref{uniqueness}, states that two different regular families on $\omega$ are not $\pi$-homeomorphic. A permutation of $\omega$ which takes $\F$ onto $\G$ and a sequence of signs induce an isometry between $X_\F$ and $X_\G$. \cite[Theorem 10]{BrechFerencziTcaciuc} shows that the converse is also true for regular families $\F$ and $\G$: any isometry between $X_\F$ and $X_\G$ is induced by a permutation of $\omega$ and a sequence of signs. Our Theorem \ref{uniqueness} shows that under any of these two equivalent assumptions - $\G$ being the image of $\F$ under a permutation of $\omega$ or $X_\F$ and $X_\G$ being isometric - $\F$ and $\G$ are just the same family. Meaning that these morphisms preserve not only the structure of the family, but the family itself. 

The crucial difference between the hypotheses of theorems \ref{perm} and \ref{uniqueness} is the assumption that the families are spreading (see Definition \ref{defFamiliesN}). A compact family can only happen to be spreading for families on $\omega$. The properties of being compact or hereditary are preserved under permutations of $\kappa$, as a permutation induces a homeomorphism between a family and its image. However, this is not the case for spreading. This lead us to suspect that being spreading is a strong assumption for the classification of the isometries between combinatorial spaces obtained in \cite{BrechFerencziTcaciuc} and this is indeed the case, as guarantees Theorem \ref{uniqueness}.

Another implication of a family being spreading is that the Cantor-Bendixson index of the singletons $\{n\}$ is a monotone nondecreasing sequence of ordinals. This fact is explored in more detail in the last section, where we also give further examples and remarks related to both the countable and the general setting. 

The paper is organized as follows. In Section 1 we introduce definitions, notations and basic facts to be used in the following ones. In Section 2 we present Theorem \ref{perm}, which characterizes the isometries on combinatorial spaces of families in a broader class than the class of regular families, including families on uncountable cardinals. In Section 3, we show that two different regular families cannot be permuted one to the other: this is Theorem \ref{uniqueness} and its proof is purely combinatorial.  We finish the paper exploring more about the topology and giving examples and remarks presented in Section 4.

\section{Preliminaries}

Given an infinite cardinal $\kappa$, we will denote by $[\kappa]^{<\omega}$ the set of all finite subsets of $\kappa$. Given $n\in\omega$ and $M\subseteq \kappa$, we denote by $[M]^n$ the collection of all subsets of $M$ of size $n$, by $[M]^{<n}$ the collection of all subsets of $M$ of size less than $n$, and by $[M]^{\leq n}$ the union $[M]^n\cup [M]^{<n}$. If $s \in [\kappa]^n$, we write $s = \{\alpha_1, \dots, \alpha_n\}_<$ to indicate that the enumeration is strictly increasing.

We introduce now some notions about families:

\begin{definition}\label{defFamilies}
A \emph{family} $\F$ on $\kappa$ is a set $\F\subseteq [\kappa]^{<\omega}$ and we will always assume that $[\kappa]^{\leq 1} \subseteq \F$. 
\begin{itemize}
    \item We say that a family $\F$ is \emph{hereditary} if it is closed under subsets: $s\in\F$ whenever $t\in\F$ and $s\subseteq t$. 
    \item A family $\F$ on $\kappa$ is a topological space when endowed with the subspace topology inherited from $2^\kappa$, where each $s \in \F$ is identified with its characteristic function and $2^\kappa$ has the product topology. 
    \item Any topological assumption about $\F$ refers to this topology. For instance, $\F$ is \emph{compact} if it is a closed subspace of $2^\kappa$.
\end{itemize}
\end{definition}

Given a family $\F$, let $\FM$ denote the set of maximal elements of $\F$ with respect to the inclusion and $\overline{\F}^\subseteq$ denote its downwards closure, that is:
$$\FM = \{s \in \F: \nexists t \in \F \text{ such that }s \subsetneq t\} \qquad \text{and}\qquad \overline{\F}^\subseteq=\allowbreak\{s: \exists t \in \F \text{ such that } s\subseteq t\}.$$ 

A permutation of $\kappa$ is simply a bijection $\pi:\kappa\longrightarrow\kappa$ and if $s \subseteq \kappa$, then $\pi[s] = \{\pi(\alpha): \alpha \in s\}$. Any permutation $\pi$ of $\kappa$ induces a homeomorphism $\hat{\pi}: 2^\kappa \longrightarrow 2^\kappa$, defined by $\hat{\pi}(x)(\alpha) = x(\pi^{-1}(\alpha))$. We chose to use here $\pi^{-1}$ instead of $\pi$ so that we have the following: if $\pi[\F]$ denotes the image of $\F$ by $\hat{\pi}$, then it coincides with $\{\pi[s]: s \in \F\}$. We say that two families $\F$ and $\G$ on $\kappa$ are $\pi$-homeomorphic if there is a permutation $\pi$ of $\kappa$ such that $\G = \pi[\F]$. 

By an isometry between Banach spaces we mean a linear bijective operator which preserves the norm. If $X$ is a Banach space, $X^*$ denotes its topological dual, $B_X$ is the unit ball and $B_{X^*}$ is the dual unit ball. 

Throughout the paper we will be concerned about conditions on two families $\F$ and $\G$ which guarantee that all or some of the following statements are equivalent and examples to show when they are not:
\begin{enumerate}[(i)]
    \item $\F = \G$.
    \item $\F$ and $\G$ are $\pi$-homeomorphic.
    \item $X_\F$ and $X_\G$ are isometric.
    \item $X_\F^*$ and $X_\G^*$ are isometric.
\end{enumerate}

Notice that these statements are decreasing in strength. (i) clearly implies (ii) and if $\pi$ is a permutation of $\kappa$ witnessing (ii), then, given any sequence $(\theta_\alpha)$ of scalars such that $|\theta_\alpha| = 1$, there is an isometry $T: X_\F \rightarrow X_\G$ satisfying $T(e_\alpha)=\theta_\alpha e_{\pi(\alpha)}$ for all $\alpha \in \kappa$, so that (ii) implies (iii). (iii) implies (iv) easily, as the adjoint operator of any isometry between two Banach spaces is an isometry between their dual spaces. Moreover, in case $T$ is such that $T(e_\alpha) = \theta_\alpha e_{\pi(\alpha)}$ for every $\alpha \in \kappa$, some sequence of scalars $(\theta_\alpha)$ and a permutation $\pi$, then clearly $T^*(e_\alpha^*) = \theta_\alpha e_{\pi^{-1}(\alpha)}^*$ for all $\alpha \in \kappa$. 

Also, we will be mostly working with hereditary families, so that conditions (i) and (ii) are clearly equivalent to conditions (i') and (ii') below, and we will use them interchangeably:
\begin{enumerate}[(i')]
    \item $\FM = \GM$.
    \item $\FM$ and $\GM$ are $\pi$-homeomorphic.
\end{enumerate}

\section{Isometries between combinatorial spaces}

The purpose of this section is to (essentially) improve  the results of \cite{BrechFerencziTcaciuc} by getting that two families are $\pi$-homeomorphic from weaker assumptions on them. Our first purpose was to replace the spreading condition by some condition which does not take the order into account. This lead us to a much more general version, for compact and hereditary families, on any infinite cardinal $\kappa$, satisfying some topological condition. The existence of nontrivial compact and hereditary families on uncountable cardinals was already known in \cite{LopezTodorcevicLargeFam}, and much more complex families were explored in \cite{BrechLopezTodorcevic}. We recall some of these examples and discuss these briefly in Section 4.

Our proof follows the path of the proof of Theorem 10 of \cite{BrechFerencziTcaciuc}. In particular, we will also make use of extreme points in our proof. Given a Banach space $X$, recall that $x$ is an extreme point of a convex set $B \subseteq X$ if there are no $x_1 \neq x_2$ in $B$ such that $x = \alpha x_1 + (1-\alpha)x_2$ for some $0<\alpha <1$. Let $Ext(B_{X^*})$ be the set of extreme points of $B_{X^*}$. 

We will use the characterization of the extreme points of $B_{X_\F^*}$ stated in \cite{GowersBlog}, and proved and used in \cite{AntunesBeanlandVietChu} and  \cite{BrechFerencziTcaciuc} in more particular cases. Actually, despite being stated for regular families in \cite[Proposition 5]{BrechFerencziTcaciuc}, a careful analysis of its proof shows that the characterization holds for any compact and hereditary family $\F$. For completeness purposes, we reproduce their proof here with the required adaptations. 

Let $\mathbb{K}$ be the field of scalars $\mathbb{R}$ or $\mathbb{C}$. Given a Banach space $X$, we say that a subset $N \subseteq B_{X^*}$ is {\em norming} if $\Vert x \Vert = \sup \{|x^*(x)|:x^* \in N\}$ for every $x \in X$, and $N$ is {\em sign invariant} if for any sign $\theta\in\mathbb{K}$, we have $\theta N=N$.
The following classical lemma was stated as Lemma 4 of \cite{BrechFerencziTcaciuc}: 

\begin{lemma} \label{norming}
Let $X$ be a Banach space over $\mathbb{K}$, and let $N\subseteq B_{X^*}$ be a sign invariant norming set for $X$. Then $B_{X^*}=\overline{conv(N)}^{w*}$.
\end{lemma}

To prove the characterization of the extreme points in our more general setting, we replace weak$^*$ convergence of sequences by the following auxiliary lemma:

\begin{lemma}\label{aux2}
Let $\F$ be a compact and hereditary family on $\kappa$. Then 
$$\{\sum_{\xi \in s} \theta_\xi e_\xi^*: s \in \F \ \text{ and }(\theta_\xi)_{\xi \in s} \text{ is a sequence of scalars with }|\theta_\xi|=1\}$$
is weak$^*$-closed.
\end{lemma}
\begin{proof}
Let
$$M = \{\sum_{\xi \in s} \theta_\xi e_\xi^*: s \in \F \ \text{ and }(\theta_\xi)_{\xi \in s} \text{ is a sequence of scalars with }|\theta_\xi|=1\}$$
and let $x^* \in \overline{M}^{w^*}$. For each $\xi \in \kappa$, notice that 
$$|x^*(e_\xi)|\in \overline{\{|y^*(e_\xi)|: y^* \in M)\}} = \{0,1\}.$$
Hence, there is $s \subseteq \kappa$ and a sequence $(\theta_\xi)_{\xi \in s}$ in $\mathbb{K}$ with $|\theta_\xi|=1$ such that $x^* = \sum_{\xi \in s} \theta_\xi e_\xi^*$. It remains to prove that $s \in \F$. 

Let us show that $[s]^{< \omega} \subseteq \F$. For each finite $t \subseteq s$, there is $y_t^* \in M$ such that $|x^*(e_\xi) - y_t^*(e_\xi)|<\frac{1}{2}$ for every $\xi \in t$. In particular, this implies that $|y_t^*(e_\xi)|=1$ for every $\xi \in t$, so that $t \subseteq \supp y_t^* \in \F$. Since $\F$ is hereditary, we get that $t \in \F$.

Finally, $s$ has to be finite since $\F$ is compact. In particular, $s \in \F$.
\end{proof}

We are now ready to reproduce the proof given in \cite[Proposition 5]{BrechFerencziTcaciuc} of the following characterization:

\begin{proposition}\label{extremepoints}
If $\F$ is a compact and hereditary family on an infinite cardinal $\kappa$, then
$$
Ext(B_{X_{\F}^{*}})=\left\{\sum_{\alpha\in s}\theta_\alpha e^*_\alpha : s\in\FM \text{ and } (\theta_\alpha)_{\alpha \in s} \text{ is a sequence of scalars with } |\theta_\alpha|=1\right\}.
$$
\end{proposition}
\begin{proof}
Take $N:=\{\sum_{\alpha\in s}\theta_\alpha e^*_\alpha : s\in \FM, |\theta_\alpha|=1\}$ and $M:=\{\sum_{\alpha\in s}\theta_\alpha e^*_\alpha : s\in \F, |\theta_\alpha|=1\}$. $M$ is clearly norming for $X_{\F}$ and sign invariant.  It follows from Lemma \ref{norming} that $B_{X_\F^*}=\overline{conv(M)}^{w^*}$. By Lemma \ref{aux2}, $M$ is $w^*$-closed, so that both $M$ and $B_{X_\F^*}=\overline{conv(M)}^{w*}$ are compact in the locally convex space $(X_\F^*,w^*)$. It follows from Milman's theorem (see \cite{Rudin}, Theorem 3.25) that every extreme point of $B_{X_\F^*}$ lies in $M$. Also, $(M \setminus N) \cap Ext(B_{X_\F^*}) = \emptyset$, as any $x \in M \setminus N$ can be written as $x = \frac{(x + e_\alpha^*) + (x - e_\alpha^*)}{2}$, where $\supp x \cup \{\alpha\} \in \F$. Since $N\subseteq Ext(B_{X_\F^*})$ we conclude that $Ext(B_{X_\F^*})=N$. 
\end{proof} 

We prove another auxiliary lemma before getting our improved version of \cite[Proposition 10]{BrechFerencziTcaciuc}

\begin{lemma}\label{aux1}
Let $\F$ be a compact and hereditary family on $\kappa$. Given $\alpha \in \kappa$, if
$\{\alpha\} \in \overline{\FM}$, then $e_\alpha^* \in \overline{Ext(B_{X_\F^*})}^{w^*}$.
\end{lemma}
\begin{proof}
Fix $\alpha \in \kappa$ such that $\{\alpha\} \in \overline{\FM}$ and let $\varepsilon >0$ and $x_1, \dots, x_n \in X_\F$. We have to show that there is $x^* \in Ext(B_{X_\F^*})$ such that $|e_\alpha^*(x_i) - x^*(x_i)| < \varepsilon$ for all $1 \leq i \leq n$.

Take $y_1, \dots, y_n \in c_{00}(\kappa)$ such that $\Vert y_i - x_i \Vert < \frac{\varepsilon}{2}$ for every $1 \leq i \leq n$ and take $F = \bigcup\{\supp y_i: 1 \leq i \leq n\}$. Since $F$ is finite and $\{\alpha\} \in \overline{\FM}$, there is $s \in \FM$ such that $\alpha \in s$ and $s \cap F \subseteq \{\alpha\}$. 

Let $x^*= \sum_{\xi \in s} e_\xi^*$ and notice that $x^* \in Ext(B_{X_\F^*})$ by Proposition \ref{extremepoints}. Also, for every $1 \leq i \leq n$, since $\supp y_i \cap s \subseteq \{\alpha\}$, we have that
$$|e_\alpha^*(y_i) - x^*(y_i)| = |\sum_{\xi \in s \setminus \{\alpha\}} e_\xi^*(y_i)| =0.$$
Finally, since $\Vert e_\alpha^*\Vert = 1$ and $\Vert x^* \Vert = 1$, we conclude that
$$\begin{array}{rl}
|e_\alpha^*(x_i) - x^*(x_i)| &
\leq |e_\alpha^*(x_i) - e_\alpha^*(y_i)| + |e_\alpha^*(y_i) - x^*(y_i)| + 
|x^*(y_i) - x^*(x_i)| \\
&\leq \Vert e_\alpha^*\Vert \cdot \Vert x_i - y_i\Vert  + \Vert x^*\Vert \cdot \Vert x_i - y_i\Vert < \varepsilon,
\end{array}$$
and this concludes the proof that $e_\alpha^* \in \overline{Ext(B_{X_\F^*})}^{w^*}$.
\end{proof}

\begin{proposition} \label{permprop}
Let $\F$ and $\G$ be compact and hereditary families on an infinite cardinal $\kappa$ with the additional property that $[\kappa]^1 \subseteq \overline{\FM}$.
If $T:X_{\F}^*\to X_{\G}^*$ is weak$^*$-weak$^*$ continuous and preserves extreme points, then for every $\alpha \in \kappa$ there is 
 $s_\alpha \in \G$ and a sequence $(\theta_\xi^\alpha)_{\xi \in s_\alpha}$ with $|\theta_\xi^\alpha|=1$ such that $T(e_\alpha^*) = \sum_{\xi \in s_\alpha} \theta_\xi^\alpha e_\xi^*$.
\end{proposition}
\begin{proof}
Fix $\alpha \in \kappa$ and since $\{\alpha\} \in \overline{\FM}$, then by Proposition \ref{extremepoints} and Lemma \ref{aux1},
$$e_\alpha^* \in \overline{Ext(B_{X_\F^*})}^{w^*} \subseteq \overline{B_{X_\F^*}}^{w^*} = B_{X_\F^*}.$$
Since $T$ is weak$^*$-weak$^*$-continuous and takes extreme points to extreme points, we get that:
$$T(e_\alpha^*) \in T[\overline{Ext(B_{X_\F^*})}^{w^*}] \subseteq \overline{T[Ext(B_{X_\F^*})]}^{w^*} \subseteq
\overline{Ext(B_{X_\G^*})}^{w^*}.$$
By Lemma \ref{aux2} for $\G$ we get that there is 
 $s_\alpha \in \G$ and a sequence $(\theta_\xi^\alpha)_{\xi \in s_\alpha}$ with $|\theta_\xi^\alpha|=1$ such that $T(e_\alpha^*) = \sum_{\xi \in s_\alpha} \theta_\xi^\alpha e_\xi^*$.
\end{proof}

We can now establish the main result of this section:

\begin{theorem} \label{perm}
Let $\F$ and $\G$ be compact and hereditary families on $\kappa$ with the additional property that 
$$[\kappa]^1 \subseteq \overline{\FM} \cap \overline{\GM}.$$
If $T:X_{\F}\to X_{\G}$ is an isometry, then there exists a permutation $\pi$ of $\kappa$ and a sequence of signs $(\theta_\alpha)_\alpha$ such that $Te_\alpha=\theta_\alpha e_{\pi(\alpha)}$ for all $\alpha\in\kappa$.
\end{theorem}
\begin{proof}
Given an isometry $T:X_{\F}\to X_{\G}$, the adjoint operator $T^*:X_{\G}^*\to X_{\F}^*$ is weak$^*$-weak$^*$ continuous and is also an isometry, so that $T^*$  takes the extreme points of $B_{X_\G^*}$ to the extreme points of $B_{X_\F^*}$. By Proposition \ref{permprop}, for each $\alpha \in \kappa$, there is $s_\alpha \in \F$ and a sequence $(\theta_\xi^\alpha)_{\xi \in s_\alpha}$ such that $|\theta_\xi^\alpha|=1$ and $T^*(e_\alpha^*) = \sum_{\xi \in s_\alpha} \theta_\xi^\alpha e_\xi^*$. The conclusion now follows similarly as Theorem 10 of \cite{BrechFerencziTcaciuc}. 
\end{proof}

Let us now compare the results of this section to those of \cite{BrechFerencziTcaciuc}. Notice that both in the previous Theorem \ref{perm} and Proposition \ref{permprop}, we have a weaker assumption on the families $\F$ and $\G$: they suppose them to be regular families and we replace the spreading condition by the strictly weaker topological condition that the $\overline{\FM}$ and $\overline{\GM}$ contain the singletons. On the other hand, we start with a slightly stronger assumption on the operator: in case of Proposition \ref{permprop} we ask $T: X_\F^* \rightarrow X_\G^*$ to be weak$^*$-weak$^*$ continuous and to preserve extreme points, while in \cite{BrechFerencziTcaciuc} they require only the latter. As a consequence, we start in Theorem \ref{perm} from an isometry $T: X_\F \rightarrow X_\G$, while they start from an isometry $T: X_\F^* \rightarrow X_\G^*$. 

Given two compact and hereditary families $\F$ and $\G$ such that $[\kappa]^1  \subseteq \overline{\FM} \cap \overline{\GM}$, we do not know the answer to the following question:

\begin{question}
Is every isometry $T: X_\F^* \rightarrow X_\G^*$ weak$^*$-weak$^*$ continuous? Or, equivalently, is every isometry $T: X_\F^* \rightarrow X_\G^*$ the adjoint operator of an isometry $S: X_\G \rightarrow X_\F$?
\end{question}

If the answer to these questions is positive, we could get a full version of Corollary 12 of \cite{BrechFerencziTcaciuc} with the weaker assumption on the families. Nevertheless, given a $\pi$-homeomorphism between two hereditary families $\F$ and $\G$ induced by a permutation $\pi$ of $\kappa$, the unique linear continuous operator taking $e_\alpha$ to $e_{\pi(\alpha)}$ is clearly an isometry between $X_\F$ and $X_\G$. Hence, we have the following corollary from our previous results:

\begin{corollary}\label{coropi}
Let $\F$ and $\G$ be compact and hereditary families on $\kappa$ such that $[\kappa]^1 \subseteq \overline{\FM} \cap \overline{\GM}$. Then TFAE:
\begin{enumerate}[(i)]
    \item[$(ii)$] $\F$ and $\G$ are $\pi$-homeomorphic.
    \item[$(iii)$] $X_\F$ and $X_\G$ are isometric.
\end{enumerate}
\end{corollary}

\section{Uniqueness of regular families under permutations}

In the previous section we improved the results of \cite{BrechFerencziTcaciuc} by getting weaker assumptions on families which are sufficient to guarantee that the combinatorial spaces are isometric iff the families are $\pi$-homeomorphic. In this section we deal with families on the countable infinite cardinal $\omega$ and improve their results in another direction. We show that two $\pi$-homeomorphic regular families are actually the same. Let us consider the following definitions:

\begin{definition}\label{defFamiliesN}
Given $s, t\in\fin$, we say that $t$ is a \emph{spread} of $s$ if there is an injective mapping $\sigma:s\longrightarrow \omega$ such that $\sigma(i)\geq i$ for every $i\in s$, and $\sigma[s]=t$; we will abbreviate this fact by saying that $\sigma[s]$ is a spread of $s$.

We call a family $\F$ \emph{spreading} if $\sigma[t]\in\F$ for every $t\in\F$ and every spread $\sigma[t]$ of $t$; and we say that a family $\F$ is \emph{regular} if it is compact, hereditary and spreading.
\end{definition}

The main purpose of this section is to show Theorem \ref{uniqueness}, which states that hereditary and spreading families on $\omega$ are invariable under permutations. In particular, different regular families on $\omega$ cannot be permuted one to the other. As a consequence of Theorem 10 of \cite{BrechFerencziTcaciuc}, we get that the combinatorial spaces of two different regular families are not isometric.

We start with the following two simple lemmas, which will be used in the proof. 

\begin{lemma}\label{otherspreading} 
Given $s\in\fin$ and a spread $\sigma[s]$ of $s$, there is a spread $\sigma'[s]$ of $s$ such that $\sigma'[s]=\sigma[s]$ and $\sigma'(i)=i$ for every $i\in s\cap\sigma[s]$.
\end{lemma}

\begin{proof} 
Let $s \in \fin$ and let $\sigma[s]$ be a spread of $s$. Define $\Delta = s \cap \sigma[s]$ and notice that, since $|s| = |\sigma[s]|$, we have that $|s \setminus \Delta| = |\sigma[s] \setminus \Delta|$. Let $s \setminus \Delta = \{m_1, \dots, m_k\}_<$ and $\sigma[s] \setminus \Delta = \{n_1, \dots, n_k\}_<$. 

Now notice that if there was $1\leq j< k$ such that $m_j<n_j<n_{j+1}<m_{j+1}$, then $\sigma[s]$ would not be a spread of $s$. Hence, $m_j < n_j$ for every $1 \leq j \leq k$. 

Define $\sigma':s\longrightarrow \omega$ by $\sigma'(i)=i$ if $i\in \Delta$ and $\sigma'(m_j)=n_j$ if $1\leq j\leq k$. Notice that $\sigma'[s] = \Delta \cup  \{n_1, \dots, n_k\} = \sigma[s]$, as wanted.  
\end{proof}

\begin{lemma}\label{lemmaclosure} 
If $\F\subseteq \fin$ is a spreading family, then so is $\overline{\F}^\subseteq$.
\end{lemma}

\begin{proof} 
Let $\F\subseteq \fin$ be a spreading family, $s\in\overline{\F}^\subseteq$ and $\sigma[s]$ a spread of $s$. We must show that $\sigma[s]\in\overline{\F}^\subseteq$. 

Let $t\in\F$ be such that $s\subseteq t$ and let $t\setminus s = \{m_1,m_2, \dots,m_k\}_<$. Then, consider integers $n_k>\dots>n_2>n_1>\max (t \cup \sigma[s])$ and define $\sigma':t\longrightarrow\omega$ by $\sigma'(j)=\sigma(j)$ if $j\in s$; and $\sigma'(m_i)=n_i$ for $1 \leq i\leq k$. It follows that $\sigma'[t]$ is a spread of $t$ and that $\sigma[s] = \sigma'[s]\subseteq \sigma'[t]$. By being $\F$ spreading, we get that $\sigma'[t]\in\F$ and, therefore, $\sigma[s]\in\overline{\F}^\subseteq$.
\end{proof}

We are now ready to prove our main result:

\begin{theorem}\label{uniqueness}
If $\F$ and $\G$ are hereditary, spreading and $\pi$-homeomorphic families\footnote{In this result we do not actually need to assume that $\F$ and $\G$ contain all singletons. If they contain, the proof for $n=1$ in the induction would be simpler, but we opted to present a proof which makes sense both under this assumption or without it.} on $\omega$, then $\F= \G$.
\end{theorem}
\begin{proof}
We prove, by induction on $n$, that for any two hereditary and spreading families $\F$ and $\G$, if there is a permutation $\pi: \omega \rightarrow \omega$ such that $\G=\pi[\F]$, then $\F \cap [\omega]^{\leq n}= \G \cap [\omega]^{\leq n}$.

For $n=1$, let 
$$I = \{i \in \omega: \{i\} \notin \F\}\text{ and } J= \{i \in \omega: \{i\} \notin \G\}.$$ Notice that, from spreading, $I$ and $J$ are initial segments of $\omega$. Moreover, $i \in I$ iff $\pi(i)\in J$, so that $|I|=|J|$. It follows that $I=J$ and, therefore, $\F\cap[\omega]^1 = \{\{i\}: i \notin I\} = \{\{i\}: i \notin J\} =\G \cap [\omega]^1$.   

Assume that the statement holds for $n$ and let us prove it for $n+1$. Given $s \in \F \cap [\omega]^n$, let
$$I_s=\{i \in \omega :  \{i \}\cup s \notin \F\}.$$
Let us prove some properties of $I_s$. It is easy to see, from spreading, that $I_s$ is an initial segment of $\omega \setminus s$. Moreover, we have the following:

\begin{claim} If $s \in \F \cap[\omega]^n$ and $t \in [\omega]^n$ is a spread of $s$ (in particular, $t \in \F$), then $|I_t| \leq |I_s|$.\end{claim} 

\begin{proof}[Proof of the claim.] 
First notice that 
$$I_t \setminus s = I_t \setminus (s\cup t) \subseteq I_s \setminus (s\cup t) = I_s \setminus t,$$ 
where the two equalities follow from the fact that $I_s \cap s = I_t \cap t = \emptyset$ and the inclusion follows from spreading of $\F$.

Second, let us show that $|I_t \cap s| \leq |I_s \cap t|$. Indeed, by Lemma \ref{otherspreading}, consider $\sigma: s \rightarrow t$ a bijection such that $\sigma(i) \geq i$ for every $i \in s$ and, additionally, $\sigma(i) = i$ for every $i \in s \cap t$. We show that if $j \in I_t \cap s$, then $\sigma(j) \in I_s \cap t$. Indeed, fix $j \in I_t \cap s$ and let $\sigma':\{\sigma(j)\} \cup s \rightarrow \{j\} \cup t$ be defined by $\sigma'(i) = \sigma(i)$ for $i \in s \setminus \{j,\sigma(j)\}$, $\sigma'(j) = j$ and $\sigma'(\sigma(j))=\sigma(j)$. $\sigma'$ is well-defined (since $\sigma (i) = i$ for $i \in s\cap t$) and it witnesses that $\{j\} \cup t$ is a spread of $\{\sigma(j)\} \cup s$.

Since $j \in I_t$, then $\sigma'[\{\sigma(j)\} \cup s] = \{j\} \cup t \notin \F$. Also, $\F$ is spreading, so that $\{\sigma(j)\} \cup s \notin \F$, that is, $\sigma(j) \in I_s$. Hence, $\sigma(j) \in I_s \cap t$ and we conclude that $|I_t \cap s| \leq |I_s \cap t|$. Summing up we get that $|I_t| \leq |I_s|$ and this concludes the proof of the claim.
\end{proof}

Now, for each $k \in \omega$, let
$$\F_{n,k} = \{s \in \F \cap [\omega]^n: |I_s| \leq k\}.$$
It follows from the claim that $\F_{n,k}$ is a spreading family, for each $k \in \omega$. 

For each $t \in \G \cap [\omega]^n$ and each $k\in\omega$, define 
$$J_t=\{i \in \omega :  \{i \}\cup t \notin \G\} \text{ and } \G_{n,k} = \{t \in \G \cap [\omega]^n: |J_t| \leq k\},$$
and get similar properties as for $I_s$ and $\F_{n,k}$. In particular, $\G_{n,k}$ is a spreading family for every $k\in\omega$.

Since both $\F_{n,k}$ and $\G_{n,k}$ are spreading and $\G_{n,k} = \pi[\F_{n,k}]$, then by Lemma \ref{lemmaclosure}, the downward closures $\overline{\F}_{n,k}^{\subseteq}$ and $\overline{\G}_{n,k}^{\subseteq}$
are hereditary and spreading families, and $\overline{\G}^\subseteq_{n,k} = \pi[\overline{\F}^{\subseteq}_{n,k}]$. It follows from the inductive hypothesis that $\G_{n,k}= \overline{\G}^\subseteq_{n,k} \cap[\omega]^n= \overline{\F}^{\subseteq}_{n,k} \cap[\omega]^n=\F_{n,k}$, for every $k\in\omega$. 

Finally, 
$$\begin{array}{ll}
\F \cap [\omega]^{n+1} & = \{\{i\}\cup s : i \notin s \text{ and } \exists k \in \omega, \ s \in \F_{n,k} \setminus \F_{n,k-1} \text{ and } i \notin \tau_k(\omega \setminus s)\}\\ 
& = \{\{i\}\cup s : i \notin s \text{ and }\exists k \in \omega, \ s \in \G_{n,k} \setminus \G_{n,k-1} \text{ and } i \notin \tau_k(\omega \setminus s)\} = \G \cap [\omega]^{n+1},
\end{array}$$
where $\tau_k: \wp(\omega)\rightarrow [\omega]^k$ associates to any set $X$ its initial segment $\tau_k(X)$ of size $k$.
\end{proof}

The following corollary follows immediately from \cite[Corollary 12]{BrechFerencziTcaciuc} and Theorem \ref{uniqueness} above, our crucial contribution being the equivalence between (i) and (ii).

\begin{corollary}\label{coroperm}
If $\F$ and $\G$ are regular families on $\omega$, then TFAE:
\begin{enumerate}[(i)]
    \item $\F = \G$.
    \item $\F$ and $\G$ are $\pi$-homeomorphic
    \item $X_\F$ and $X_\G$ are isometric.
    \item $X_\F^*$ and $X_\G^*$ are isometric.
\end{enumerate}
\end{corollary}

\section{Examples and further remarks}

The purpose of this section is to give some examples and prove several facts which put our main results into a bigger picture. 
In particular, a topological approach is useful.

For most of the results in this section we will use the Cantor-Bendixson rank. We recall that if $X$ is a topological space, then $X'=\allowbreak\{x\in X: x\text{ is not isolated in }X\}$ and given an ordinal $\alpha$, the $\alpha-$th Cantor-Bendixson derivative of $X$ is defined recursively by:
 $$\begin{array}{rcl}
 X^{(0)} & = & X\\
 X^{(\beta+1)}& = & (X^{(\beta)})'\\
 X^{(\lambda)}& = & \bigcap_{\beta<\lambda} X^{(\beta)}, \text{ if }\lambda>0 \text{ is a limit ordinal.}
 \end{array}$$
The Cantor-Bendixson rank of $X$ is the minimal ordinal such that $X^{(\alpha)} = X^{(\alpha+1)}$ and it is denoted by $\rk(X)$.

If $\F\subseteq [\kappa]^{<\omega}$ is a  compact family, then its Cantor-Bendixson index is the smallest ordinal $\alpha$ such that $\F^{(\alpha)}=\emptyset$ and it is therefore a successor ordinal. Moreover, $\rk(\F) < \kappa^+$. 

\subsection{Injective Cantor-Bendixson index}

Given a compact family $\F\subseteq [\kappa]^{<\omega}$, we can define a function $\rk_\F: \F \rightarrow \kappa^+$ which associates to $s\in\F$ the biggest ordinal $\beta$ for which $s\in\F^{(\beta)}$. Notice that the rank of a point coincides with the rank of its image under a homeomorphism between two topological spaces. So, in particular, the rank of a singleton coincides with the rank of its image under a $\pi$-homeomorphism between two families. 

When a compact family $\F$ on $\omega$ is spreading, the corresponding function $\rk_\F: [\omega]^1 \rightarrow \omega_1$ is nondecreasing. In particular, for the Schreier family
$$\Sh:= \{s \in [\omega]^{< \omega}: |s| \leq \min(s)+1\} \cup \{ \emptyset\}$$
(known to be regular) $\rk_\Sh(\{n\})=n$ for every $n \in \omega$, so that $\rk_\Sh \restriction [\omega]^1$ is injective. This imposes restrictions on which permutations induce a homeomorphism of a given family: 

\begin{remark}\label{propinjective}
If $\mathcal{F}$ is a compact family on $\kappa$ such that $\rk_\F\restriction [\kappa]^1$ is injective, then the only permutation $\pi$ of $\kappa$ such that $\pi[\F]=\F$ is the identity map.
\end{remark}

If follows from this remark and Theorem \ref{perm} that if, moreover, $[\kappa]^1 \subseteq \overline{\FM}$, then every isometry of $X_\F$ is given by a change of signs of the elements of its basis. This generalizes \cite[Theorem 15]{BrechFerencziTcaciuc}. 

Also, for families which have stricly increasing rank function, we can easily get the conclusion of Theorem \ref{uniqueness}:

\begin{proposition}\label{Sunicity}
If $\F$ and $\G$ are $\pi$-homeomorphic compact families on $\kappa$ such that $\rk_\F$ is strictly increasing and $\rk_\G$ is nondecreasing, then the only permutation of $\kappa$ inducing an homeomorphism between them is the identity map. In particular, $\F=\G$.
\end{proposition}
\begin{proof} 
Let $\F$ and $\G$ be as in the hypothesis, and let $\pi$ be a permutation of $\kappa$ such that $\G = \pi[\F]$. We will prove that $\pi$ is increasing and since it is a permutation, it has to be the identity map. Assume towards a contradiction that there are $\alpha < \beta < \kappa$ such that $\pi(\alpha)\geq \pi(\beta)$. Then $$\rk_\F(\{\alpha\}) = \rk_\G(\{\pi(\alpha)\}) \geq \rk_\G(\{\pi(\beta)\}) = \rk_\F(\{\beta\}),$$
contradicting the fact that $\rk_\F$ is strictly increasing.
\end{proof}

As an application, we get an example of distinct compact and hereditary families which are $\pi$-homeomorphic. It guarantees that Theorem \ref{uniqueness} cannot be improved by dropping the spreading condition:

\begin{example}\label{nonvoidconditions}
 Let $\pi \neq Id$ be any permutation of $\omega$ and $\F = \pi[\Sh]$. From Remark \ref{propinjective} we get that $\F \neq \Sh$. Since $\F$ is $\pi$-homeomorphic to $\Sh$, $\F$ is also compact and hereditary and it follows from Corollary \ref{coropi} that $X_\Sh$ and $X_\F$ are isometric. 
\end{example}

\subsection{More on spreading families}

In the context of compact spreading families on $\omega$, the hypotheses of Proposition \ref{Sunicity} are almost complete, the only missing one being that the rank of one of the families has to be not only nondecreasing, but strictly increasing. The next example shows that Proposition \ref{Sunicity} does not necessarily hold when $\rk_\F$ is not strictly increasing. It also shows that two families may be $\pi$-homeomorphic with only one of them being spreading.

\begin{example} Consider $\F=[\omega]^{\leq 2}\setminus\big\{\{2,3\}\big\}$ and $\G=[\omega]^{\leq 2}\setminus\big\{\{1,2\}\big\}$. It is clear that $\F$ is compact, that $\rk_\F(\{n\})=1$ for every $n\in\omega$, and that $\G$ is spreading. The permutation $\pi:\omega\longrightarrow\omega$ given by $\pi(1)=3$, $\pi(2)=1$, $\pi(3)=2$, and $\pi(n)=n$ for every $n>3$ witnesses that $\G=\pi[\F]$. Nevertheless, $\F\neq\G$.\end{example}

 It is well known that for every limit ordinal $\alpha<\omega_1$ there is a compact family $\F$ on $\omega$ with rank $\alpha+1$ such that $\rk_\F$ is strictly increasing; for example, the Schreier families (see \cite{ArgyrosStevo}) or the closure of an $\alpha$-uniform front (introduced in \cite{PudlakRodl}). We should also mention the existence of a non hereditary and non spreading compact family $\F$ on $\omega$ with strictly increasing rank: let for example 
 $$\F=\{s\in\Sh:\{n,n+1\}\text{ is not an initial segment of $s$, for every }n\in\omega\};$$ 
 then $\F$ is not hereditary nor spreading since $\{3,5,6\}\in\F$ but $\{5,6\},\{4,5,6\}\notin\F$. Thus, although we cannot apply Theorem \ref{uniqueness} here to conclude that any spreading family $\pi$-homeomorphic to $\F$ is $\F$ itself, we can instead get that conclusion from Proposition \ref{Sunicity}.

As a consequence of Theorem \ref{uniqueness}, if $\F$ and $\G$ are two different hereditary and spreading families, they are not $\pi$-homeomorphic. A natural question to ask is whether $\F$ and $\G$ turn out to be $\pi$-homeomorphic if in addition we ask them to be homeomorphic. The following example answers negatively this question.  

\begin{example}\label{nobanachstone}
Let $\F=[\omega]^{\leq 1}\cup [\omega\setminus \{1\}]^{2}$ and $\G=[\omega]^{\leq 2}$. Consider the mapping $\varphi:\F\longrightarrow\G$ given by $\varphi(\emptyset)=\emptyset$, $\varphi(\{1\})=\{1,2\}$, $\varphi(\{n\})=\{n-1\}$ if $n>1$, $\varphi(\{2,m\})=\{1,m\}$ for $m>2$, and $\varphi(\{n,m\}_<)=\{n-1,m-1\}$ for every $n>2$. It is clear that $\varphi$ is a homeomorphism. Nevertheless, there is no permutation $\pi:\omega \longrightarrow \omega$ such that $\G = \{\pi[s]: s \in \F\}$. Indeed, if such a permutation $\pi$ exists, then $\pi(\{1\})$ should be an isolated point in $\G$ with size 1, which is impossible. Hence, $\F$ and $\G$ are homeomorphic regular families which are not $\pi$-homeomorphic. In particular, $X_\F$ and $X_\G$ are not isometric.
\end{example}

Since each regular family determines a class of compact and hereditary families on $\omega$ which have isometric combinatorial spaces, it would be nice if every compact and hereditary family on $\omega$ would be $\pi$-homeomorphic to a regular one, so that regular families would classify all isometric-types of separable combinatorial spaces. However, this is not the case, as shows the following example.

\begin{example} 
Consider $\F=[\N]^{\leq 2}\setminus\big\{\{n,n+1\}:n\in\N\big\}$ which is clearly compact, hereditary and not spreading. We claim that $\F$ is not $\pi$-homeomorphic to any regular family. Indeed, assume towards a contradiction that there is a regular family $\G$ and a permutation of $\omega$ such that $\G=\pi[\F]$. Then, $\{\pi(n),\pi(n+1)\}\notin\G$ for every $n<\omega$. 

We will now see that $\pi(n)$ and $\pi(n+1)$ is one the successor of the other, for every $n<\omega$. Indeed, if there are $n,k<\omega$ such that $\pi(n)<k<\pi(n+1)$, let $p<\omega$ be such that $\pi(p)=k$. Since $p\neq n, n+1$, then $\{n,p\}\in\F$ and $\{\pi(n),k\}\in\G$. By being $\G$ spreading, we have that $\{\pi(n),\pi(n+1)\}\in\G$, which is a contradiction. Analogously, we see that there are not $n,k<\omega$ such that $\pi(n+1)<k<\pi(n)$.

Let $m<\omega$ be such that $\pi(m)=1$. Then, $\pi(m+1)=2$ as it is the successor of $\pi(m)$. It follows that $\pi[\{m,m+1,m+2,...\}]=\omega$. Thus, $\pi$ is the identity map, contradicting that $\F$ is not spreading.
\end{example}

\subsection{Examples in the uncountable setting}

Our purpose now is to explore some examples of families  on a given uncountable cardinal $\kappa$ which satisfy the hypotheses of Theorem \ref{perm}. It is clear that $\F = [\kappa]^{\leq n}$ is such a family, but we consider this to be somewhat trivial, as the norm looks very much like the supremum norm and any permutation of $\kappa$ fixes $\F$. 

On the other hand, if we assume $\F$ on some uncountable cardinal $\kappa$ not to be contained in any $[\kappa]^{\leq n}$, then it cannot be spreading: let $(s_n)_n \subseteq \F$ with $|s_n| \geq n$ and fix $\alpha \in \kappa$ such  $s_n \subseteq \alpha$ for every $n \in \omega$. Spreading and hereditary would imply that the intervals $[\alpha, \alpha + n] \in \F$. In turn, this would contradict compactness, since this sequence converges to $[\alpha, \alpha + \omega)$, which is infinite and does not belong to $\F$.

\begin{example}\label{S_omega_kappa}
Given a cardinal $\kappa$, let $\mathcal{S}(\kappa)$ be the family formed by copies of the Schreier family $\mathcal{S}$ in each interval of the form $[\lambda, \lambda + \omega)$, where $\lambda <\kappa$ is a limit ordinal (or zero). Formally,
$$\mathcal{S}(\kappa) = \{s \in [\kappa]^{<\omega}: s \subseteq [\lambda , \lambda + \omega) \text{ for some limit ordinal $\lambda<\kappa$ and } |s| \leq \min \{n \in \omega: \lambda + n \in s \}+1\}.$$
\end{example}

Notice that $\mathcal{S}(\kappa)$ is a compact hereditary family such that $\{\alpha\} \in \overline{\mathcal{S}(\kappa)^{MAX}}$ for each $\alpha\in\kappa$, so that it satisfies the hypothesis of Theorem \ref{perm}. Moreover, $\mathcal{S}(\kappa)$ contains arbitrarily large (finite) elements. The corresponding combinatorial space $X_{\mathcal{S}(\kappa)}$ is isometric to an $\ell_\infty$-sum of $\kappa$ many copies of $X_\mathcal{S}$. 

To get an example which is a bit more involving we use trees. A tree is a partially ordered set $(T, \leq)$ such that for every $f \in T$, the set $\{g \in T: g \leq f\}$ is well-ordered. Hence, we can define the height $ht_T(f)$ of $f$ in $T$ as the only ordinal which is order-isomorphic to $\{g \in T: g \leq f\}$. The height $ht(T)$ of $T$ is the supremum of all $ht_T(f)$, $f \in T$. 
A chain in $T$ is a chain with respect to $\leq$, that is, a subset $C \subseteq T$ such that for every $f,g \in C$, either $f \leq g$ or $g \leq f$. A branch in $T$ is a maximal chain in $T$.

Finally, we have the following proposition, which is implicit in \cite[Lemma 2.32]{BrechLopezTodorcevic}:

\begin{proposition}\label{trees}
Let $\kappa$ be an infinite cardinal and let $T$ be a tree of height $\kappa$. If there is a compact and hereditary family $\mathcal{F}$ on $\kappa$ such that $[\kappa]^1 \subseteq \overline{\FM}$ , then there is such a family on $\lambda= |T|$.
\end{proposition}
\begin{proof}
First notice that the existence of a compact and hereditary family on some cardinal $\kappa$ is equivalent to the existence of a compact and hereditary family on some index set $I$ of cardinality $\kappa$, since being compact and hereditary doesn't depend on the order.

Hence, without loss of generality, we may assume that $\F$ is a compact and hereditary family on $\kappa+1$ (instead of $\kappa$) such that $[\kappa +1]^1 \subseteq \overline{\FM}$. 
We define
$$\G = \{ C \subseteq T: C \text{ is a chain  and } \{ht_T(f): f \in C\} \in \F\}.$$
 It is easy to see that $\G$ is a compact and hereditary family on $T$. Given $f \in T$, we fix a branch $C$ which contains $f$ and, given $(s_n)_n \subseteq \FM$ converging to $\{ht_T(f)\}$, it is easy to built $(t_n)_n$ converging to $\{f\}$, where each $t_n \subseteq C$ and $\{ht_T(g): g \in t_n\} = s_n$. Clearly each $t_n \in \GM$, which concludes the proof that $[T]^1 \subseteq \overline{\GM}$. From the first paragraph of the proof, it follows that the desired family on $\lambda=|T|$ exists.
\end{proof}

Example \ref{S_omega_kappa} can be seen as a particular case of a family given by Proposition \ref{trees}, taking $T$ to be $\kappa$-many incomparable copies of $\omega$. Also, if we have a family on $\kappa$, we get a family on $2^\kappa$, by taking $T$ the complete binary tree of height $\kappa$. This way we get more interesting families for every cardinal below the first inaccessible cardinal.

Motivated by Remark \ref{propinjective} and Proposition \ref{Sunicity}, we asked ourselves if there is a compact hereditary family $\mathcal{F}$ on $\omega_1$ such that $rk(\{\alpha\}) = \alpha$ for every $\alpha < \omega_1$. The following example answers this question positively:

\begin{example} 
Given $\alpha<\omega_1$, consider $\F_\alpha$ a compact and hereditary family on $\omega$ with $\rk(\F_\alpha)=\alpha+1$ (the existence of these families is guaranteed, for example, in \cite{ArgyrosStevo}).

Given a limit ordinal $\lambda<\omega_1$, and $s\in[\omega_1]^{<\omega}$ such that $s\subseteq [\lambda,\lambda+\omega)$, we will denote by $s-\lambda$ the set $\{\gamma-\lambda:\gamma\in s\}$ which belongs to $[\omega]^{<\omega}$. In a similar way as we have defined $\Sh(\kappa)$ in Example \ref{S_omega_kappa}, we now define the families $\F_\alpha(\kappa)$, for $\kappa$ a tail of $\omega_1$ and $\alpha<\omega_1$, by
$$\F_\alpha(\kappa)=\{s\in[\kappa]^{<\omega}:\text{ there is $\lambda<\kappa$ a limit ordinal or zero such that $s\subseteq [\lambda,\lambda+\omega)$ and $s-\lambda\in\F_\alpha$}\}.$$

Notice that $\F_\alpha(\kappa)$ is a compact and hereditary family on $\kappa$ with $\rk(\F_\alpha(\kappa))=\alpha+1$. Consider $$\F:=\Sh\cup\bigcup_{\omega\leq\alpha<\omega_1}\big\{\{\alpha\}\cup s:s\in\F_\alpha(\omega_1\setminus\alpha),\text{ and }\alpha<\gamma\text{ for every }\gamma\in s\big\}.$$

It is clear that $\F$ is a compact and hereditary family on $\omega_1$. Moreover, notice that $\rk_\F(\{\alpha\})=\rk_{\F_\alpha(\omega_1\setminus\alpha)}(\emptyset)=\rk(\F_\alpha(\omega_1\setminus\alpha))-1=\alpha$, for every $\omega\leq\alpha<\omega_1$.
\end{example}


\begin{thebibliography}{1}

\bibitem{AntunesBeanlandVietChu} L. Antunes, K. Beanland, and H. Viet Chu, \emph{On the geometry of higher order Schreier spaces}, preprint.

\bibitem{ArensKelley}
R. F. Arens, J. L. Kelley, \emph{Characterization of the space of continuous functions over a compact Hausdorff space}, Trans. Amer. Math. Soc. 62 (1947), 499–508.

\bibitem{ArgyrosStevo}
S. A.~Argyros, S.~Todorcevic,
\newblock {\em Ramsey methods in analysis},
\newblock Advanced Courses in Mathematics. CRM Barcelona. Birkhäuser Verlag, Basel, 2005.

\bibitem{BrechFerencziTcaciuc} C. Brech, V. Ferenczi and A. Tcaciuc, \emph{Isometries of combinatorial Banach spaces}, accepted to Proc. Amer. Math Soc.

\bibitem{BrechLopezTodorcevic} C. Brech, J. Lopez-Abad, S. Todorcevic, \emph{Homogeneous families on trees and subsymmetric basic sequences}, Adv. Math. 334 (2018), 322-388.

\bibitem{GowersBlog} W. Gowers,   \emph{Must an ``explicitly defined" Banach space contain $c_0$ or $l_p$?}, Feb. 17, 2009, Gowers's Weblog: Mathematics related discussions.

\bibitem{Kadison} E. V. Kadison, \emph{Isometries of Operator Algebras}, Annals of Mathematics Second Series, Vol. 54, No. 2 (Sep., 1951), pp. 325-338.

\bibitem{LindenstraussTzafriri} J.~Lindenstrauss, L.~Tzafriri, \emph{Classical Banach spaces. I. Sequence spaces}, Ergebnisse der Mathematik und ihrer Grenzgebiete, Vol. 92. Springer-Verlag, Berlin-New York, 1977.

\bibitem{LopezTodorcevicLargeFam}
J. Lopez-Abad, S. Todorcevic, \emph{Positional graphs and conditional structure of weakly null sequences}, Adv. Math. 242 (2013), 163–186.

\bibitem{PudlakRodl}
P. Pudlák, V. Rödl, \emph{Partition theorems for systems of finite subsets of integers}, Discrete Math. 39 (1) (1982) 67--73.

\bibitem{Rudin} W. Rudin, \emph{Functional analysis}, International Series in Pure and Applied Mathematics, McGraw-Hill, Inc., New York, 1991.

\bibitem{Semadeni}  Z. Semadeni, \emph{Banach spaces of continuous functions}, Monografie Matematyczne, Tom 55, Pa\'nstwowe Wydawnictwo Naukowe, Warsaw, 1971.

\end{thebibliography}
\end{document}